\documentclass[final,3p]{elsarticle}

\usepackage{amsmath,amsthm,amssymb,mathrsfs}
\usepackage{enumerate}
\usepackage{slashed}
\usepackage{txfonts,dsfont}
\usepackage{aliascnt}
\usepackage[breaklinks,colorlinks]{hyperref}

\makeatletter
\newcommand{\autorefcheckize}[1]{%
  \expandafter\let\csname @@\string#1\endcsname#1%
  \expandafter\DeclareRobustCommand\csname relax\string#1\endcsname[1]{%
    \csname @@\string#1\endcsname{##1}\wrtusdrf{##1}}%
  \expandafter\let\expandafter#1\csname relax\string#1\endcsname
}
\makeatother

\theoremstyle{plain}

\newtheorem{theorem}{Theorem}[section]

\newaliascnt{lem}{theorem}
\newtheorem{lem}[lem]{Lemma}
 \aliascntresetthe{lem}

\newaliascnt{cor}{theorem}
\newtheorem{cor}[cor]{Corollary}
 \aliascntresetthe{cor}

\newaliascnt{prop}{theorem}

 \aliascntresetthe{prop}

\newtheorem{conj}{Conjecture}

\theoremstyle{remark}

\theoremstyle{definition}

\newtheorem{eg}{Example}[section]

\numberwithin{equation}{section}

\newcommand{\abs}[1]{\left\lvert#1\right\rvert}
\newcommand{\set}[1]{\left\{#1\right\}}
\newcommand{\hin}[2]{\left\langle#1,#2\right\rangle}

\newcommand*{\To}{\longrightarrow}
\newcommand*{\Rmn}[1]{\uppercase\expandafter{\romannumeral#1}}
\newcommand*{\dif}{\mathop{}\!\mathrm{d}}

\journal{XXX}

\allowdisplaybreaks

\begin{document}

\begin{frontmatter}

\title{Rigidity theorems for minimal Lagrangian surfaces with  Legendrian capillary boundary\tnoteref{ls}}

\author[whu1,whu2]{Yong Luo}
\ead{yongluo@whu.edu.cn}

\author[whu1,whu2]{Linlin Sun\corref{sll}}
\ead{sunll@whu.edu.cn}

\address[whu1]{School of Mathematics and Statistics, Wuhan University, 430072 Wuhan, Hubei, China}
\address[whu2]{Hubei Key Laboratory of Computational Science, Wuhan University, 430072 Wuhan, Hubei, China}


\cortext[sll]{Corresponding author.}

\tnotetext[ls]{After our proofs of \autoref{mainthm:2} was completed, we learned from Professor Guofang Wang  that he had also noticed that the boundary of minimal Lagrangian surfaces in $\mathbb{B}^4$ with Legendrian capillary boundary on $\mathbb{S}^3$ are great circles and then he got an idea to prove Conjecture 1 which could be quite standard. The authors would like to thank him for the encouragement for us to complete and submit our paper. Many thanks to Dr.  Qing Cui, Jiabin Yin and  Jingyong Zhu for their interests in this paper and discussions. This research is partially supported by the National Natural Science Foundation of China  (Grant Nos. 11971358, 11801420) and the Youth Talent Training Program of Wuhan University. The second author also would like to express his gratitude to Professor J\"urgen Jost for his invitation to MPI MIS for their hospitality.}

\begin{abstract}
In this note, we study minimal Lagrangian surfaces in $\mathbb{B}^4$ with Legendrian capillary boundary on $\mathbb{S}^3$. On the one hand, we prove that any minimal Lagrangian surface in $\mathbb{B}^4$ with Legendrian free boundary on $\mathbb{S}^3$ must be an equatorial plane disk. One the other hand, we show that any annulus type minimal Lagrangian surface in $\mathbb{B}^4$ with Legendrian capillary boundary on $\mathbb{S}^3$ must be congruent to one of the Lagrangian catenoids. These results confirm the conjecture proposed by Li, Wang and Weng (Sci. China Math., 2020). 
 
\end{abstract}

\begin{keyword}
Minimal Lagrangian surface \sep Legendrian capillary boundary
 \sep Lagrangian catenoid

 \MSC[2010] 53C24\sep 53C42

\end{keyword}

\end{frontmatter}

\section{Introduction}
Let $\mathbb{C}^n=\mathbb{R}^{2n}$ be the standard complex plane with its canonical K\"ahler form $\omega$ and almost complex structure $J$. Let $\mathbb{S}^{2n-1}$ be the $ (2n-1)$-dimensional unit sphere with standard Sasakian structure. Then an $n$-dimensinal submanifold $\Sigma^n$ in $\mathbb{C}^n$ is called a Lagrangain submanifold if $JT\Sigma^n=T^{\bot}\Sigma^n$, where $T^{\bot}\Sigma^n$ denotes the normal space of $\Sigma^n$ in $\mathbb{C}^n$, and an $(n-1)$ dimensional submanifold $K^{n-1}$ in $\mathbb{S}^{2n-1}$ is called a Legendrian submanifold if $\mathbf{R}\perp TK^{n-1}$, where $\mathbf{R}$ is the Reeb field of $\mathbb{S}^{2n-1}$ with $\mathbf{R}(x)=Jx$ for every $x\in \mathbb{S}^{2n-1}$.

It is well known that Lagrangian submanifolds in a complex space form have many similarities with hypersurfaces in a real space form. Recently, inspired by the study of capillary hypersurfaces $M$ in $\mathbb{B}^{n+1}\subset\mathbb{R}^{n+1}$, which have constant mean curvature, non-empty boundary such that $\mathring{M}\subset \mathring{\mathbb{B}}^{n+1}$ and $\partial M\subset\partial\mathbb{B}^{n+1}=\mathbb{S}^n$, which intersect $\partial\mathbb{B}^{n+1}$ with a constant angle, Li, Wang and Weng \cite{LWW} initiated the very interesting study of Lagrangian submanifolds with Legendrian capillary boundary in $\mathbb{B}^{2n}\subset\mathbb{C}^n$. 

First let us recall some definitions introduced in \cite{LWW}. Let $x: \Sigma^n \to\mathbb{B}^{2n}$ be a Lagrangian submanifold with $\partial\Sigma^n\subset\partial\mathbb{B}^{2n}=\mathbb{S}^{2n-1}$ being a Legendrian submanifold. Li, Wang and Weng observed that the unit normal $\nu$ at $x\in\partial\Sigma^n\subset\Sigma^n$ lies in the plane spanned by $x$ and $ Jx$, i.e. there exists a $\theta\in[0, \pi)$ such that
$$\nu=\sin\theta x+\cos\theta Jx.$$
The angle $\theta$ is called a contact angle and $\Sigma^n$ is called a Lagrangian submanifold with
Legendrian capillary boundary (or simply capillary Lagrangian submanfold), if the
contact angle is a local constant. When $\theta=\frac{\pi}{2}$, $\Sigma^n$ is called  a Lagrangian submanifold with Legendrian free boundary, or a free boundary Lagrangian submanifold.

When $n=2$, typical examples of minimal Lagrangian surfaces in $\mathbb{B}^4$ with Legendrian capillary boundary are the equatorial plane disk and the Lagrangian catenoids, as discussed in \cite{LWW} (see also \autoref{exam}). Note that the contact angle for the equatorial plane disk is $\frac{\pi}{2}$, but the contact angle for Lagrangian catenoids are constants which are not equal to $\frac{\pi}{2}$. Li, Wang and Weng \cite{{LWW}} got the following Nitche (or Hopf) type rigidity theorem.
\begin{theorem}[Li, Wang and Weng]\label{LWW}
Given $D:=\set{(x_1, x_2): x_1^2+x_2^2\leq1}$.  Let $x: D\To \mathbb{B}^4$ be a (branched) minimal Lagrangian surface with Legendrian capillary boundary on $\mathbb{S}^3$. Then $x(D)$ is an equatorial plane disk.
\end{theorem}
This theorem is the Lagrangian counterpart of related results for capillary surfaces in $\mathbb{B}^n$ by Nitsche \cite{Nit}, Ros and Souam \cite{RS} and Fraser and Schoen \cite{FS}. Then they conjectured that: 
\begin{quote}
    There is no annulus type minimal Lagrangian surface with Legendrian free boundary.
\end{quote} 
Moreover, they wrote down the following (\cite[Conjecture 2.16]{LWW}).

\begin{conj}\label{conj}
Any embedded annulus type minimal Lagrangian surface
with Legendrian capillary boundary on $\mathbb{S}^3$ is one of the Lagrangian catenoids.
\end{conj}

This conjecture is the Lagrangian counterpart of the conjecture  for free boundary minimal surfaces in $\mathbb{B}^3$ proposed by Fraser and Li \cite{FL}.
\begin{conj}[Fraser-Li]\label{conj2}
The critical catenoid is the unique embedded free boundary minimal annulus in $\mathbb{B}^3$.
\end{conj}

In this paper, we first show  that Lagrangian minimal surfaces in $\mathbb{B}^4$ with Legendrian free boundary must be an equatorial plane disk (\autoref{mainthm:1}), which extends  \autoref{LWW} in the Legendrian free boundary case and confirms the statement:
\begin{quote}
    There is no annulus type minimal Lagrangian surface with Legendrian free boundary.
\end{quote}
Finally, we give an affirmative answer to \autoref{conj}. Actually, we prove  that \autoref{conj} is true without the embeddedness assumption (\autoref{mainthm:2}).

As it is well known, hypersurfaces in a real space form have many similarities with Lagrangian submanifolds in a complex space from, and many rigidity results for minimal hypersurfaces in a real space form have their Lagrangian counterparts. But according to our knowledge,   rigidity results in the Lagrangian submanifolds case are always much more complicated and their proofs (if they exist) need more job. Consequently, although some rigidity results are true for minimal hypersurfaces in a real space form, their Lagrangian counterparts are still open. For example  Brendle \cite{Bre} proved the longstanding Lawson's conjecture, which states that the Clifford torus is the unique embedded minimal tori in $\mathbb{S}^3$. But its Lagrangian counterpart, that is if embedded minimal Lagrangian tori in $\mathbb{CP}^2$ are given by the examples constructed by Haskins \cite{Has} with certain symmetry (see also \cite{CU, Jo}), remains widely  open. Another example is the conjecture given by the authors  (\cite[Conjecture 1]{LS}) of this paper on the first pinching constant of closed minimal Lagrangian submanifolds in $\mathbb{CP}^n$, while the case of closed minimal hypersurfaces was established by Simons \cite{Li}, Chern, do Carmo and Kobayashi \cite{CCK} and Lawson \cite{La}.

Bewaring of this, it would be a surprise for us to see that though Fraser and Li's conjecture, i.e.  \autoref{conj2}, remains open, but its Lagrangian counterpart, i.e.  \autoref{conj}, could be verified. The above mentioned Nitsche (or Hopf) type rigidity results for capillary surfaces \cite{FS, Nit, RS} and  \autoref{LWW} were proved by the technique of Hopf's holomorphic cubic form. While in our proof of \autoref{conj} we use simultaneously Hopf's holomorphic  cubic form and a maximum principle for surfaces with boundary, which is quite subtle. The main observation is that, the boundary of a minimal Lagrangian submanifold in $\mathbb{B}^{2n}$ with Legendrian capillary boundary on $\mathbb{S}^{2n-1}$ is still minimal (see \autoref{lem:2.2}), which enable us to use the maximum principle. It would be very interesting to see if this method is workable for Fraser  and Li's conjecture, by exploring more boundary properties of the critical catenoid. Here we would like to point out that Li \cite{Li} observed that by a Bj\"orling-type uniqueness result for free boundary minimal surfaces of Kapouleas and Li \cite{KL}, to prove Fraser and Li's conjecture, it suffices to show that one of the boundary components of the minimal annulus is rotationally symmetric. We invite the readers who desire more information on Fraser and Li's conjecture to consult the recent excellent surveys by Li \cite{Li} and Wang and Xia \cite{XW} and references therein. See also Fraser and Schoen \cite{FS2} for a very deep characterization of the critical cateniod.

The rest of this paper is organized as follows. In section 2 we give some properties of the Legendrian boundary and contact angle. Main results of this paper and their proofs are given in section 3.

\section{Properties of the Legendrian boundary and contact angle}

Let $x:\Sigma^{n}\To\mathbb{B}^{2n}$ be an immersed  Lagrangian submanifold with boundary $\partial \Sigma^{n}$ on the unit round sphere $\mathbb{S}^{2n-1}$.  Let  $\nu$ be the unit outward normal vector field of  $\partial \Sigma^{n}\hookrightarrow \Sigma^{n}$. Since $\Sigma^n$ is a Lagrangian submanifold of $\mathbb{B}^{2n}$, on the boundary we have the following orthogonal decomposition
\begin{align*}
T\mathbb{B}^{2n}\vert_{\partial \Sigma^{n}}=&T\Sigma^{n}\vert_{\partial \Sigma^{n}}\oplus T^{\bot}\Sigma^{n}\vert_{\partial \Sigma^{n}}\\
=&T\Sigma^{n}\vert_{\partial \Sigma^{n}}\oplus JT\Sigma^{n}\vert_{\partial \Sigma^{n}}\\
=&T\partial \Sigma^{n}\oplus JT\partial \Sigma^{n}\oplus \mathrm{span}\set{\nu,J\nu}.
\end{align*}
Notice that
\begin{align*}
T\mathbb{B}^{2n}\vert_{\partial \Sigma^{n}}=&T\mathbb{S}^{2n-1}\vert_{\partial\Sigma^{n}}\oplus\mathrm{span}\set{x}\\
=&T\partial \Sigma^{n}\oplus T^{\bot}\left(\partial \Sigma^{n}\hookrightarrow\mathbb{S}^{2n-1}\right)\oplus\mathrm{span}\set{x}.
\end{align*}
Therefore $\partial\Sigma^{n}$ is a Legendrian submanifold of $\mathbb{S}^{2n-1}$ if and only if
\begin{align*}
T^{\bot}\left(\partial \Sigma^{n}\hookrightarrow\mathbb{S}^{2n-1}\right)=JT\partial \Sigma^{n}\oplus\mathrm{span}\set{Jx},
\end{align*}
if and only if
\begin{align*}
\mathrm{span}\set{\nu,J\nu}=\mathrm{span}\set{x,Jx},
\end{align*}
which is equivalent to that
\begin{align}\label{eq:nu}
\nu=\sin\theta x+\cos\theta Jx,
\end{align}
where $\theta:\partial \Sigma^{n}\To[0,\pi)$ is a smooth function. The angle $\theta$ is called a contact angle.

Let $\mathbf{B}, \mathbf{B}^{\Sigma}$ and $\mathbf{B}^{\partial}$  be the second fundamental form of the isometric  immersions $\Sigma^{n}\hookrightarrow\mathbb{B}^{2n}, \partial \Sigma^{n}\hookrightarrow\Sigma^{n}$ and $\partial\Sigma^{n}\hookrightarrow\mathbb{S}^{2n-1}$ respectively. Let $\mathbf{H}, \mathbf{H}^{\Sigma}$ and $\mathbf{H}^{\partial}$  be the mean curvature vector of the isometric  immersions $\Sigma^{n}\hookrightarrow\mathbb{B}^{2n}, \partial \Sigma^{n}\hookrightarrow \Sigma^{n}$ and $\partial \Sigma^{n}\hookrightarrow\mathbb{S}^{2n-1}$ respectively.  Finally, let $\bar\nabla, \nabla$ and $\nabla^\partial$ be the Levi-Civita connections on $\mathbb{B}^{2n}, \Sigma^{n}$ and $\partial \Sigma^{n}$ respectively. 
\begin{lem}
For all $X, Y, Z\in T\partial \Sigma^{n}$,
\begin{align}
\mathbf{B}^{\Sigma}\left(X,Y\right)=&-\sin\theta \hin{X}{Y}\nu,\label{eq:1.2}\\
\hin{\mathbf{B}\left(X,Y\right)}{J\nu}=&\cos\theta\hin{X}{Y},\label{eq:1.3}\\
\hin{\mathbf{B}\left(X,Y\right)}{JZ}=&\hin{\mathbf{B}^\partial\left(X,Y\right)}{JZ}.\label{eq:1.4}
\end{align}
Moreover,
\begin{align}
\nabla^{\partial}\theta=J\mathbf{B}\left(\nu,\nu\right).\label{eq:1.5}
\end{align}
\end{lem}
\begin{proof}
On the one hand, the isometric immersions $\partial \Sigma^{n}\hookrightarrow\Sigma^{n}\hookrightarrow\mathbb{B}^{2n}$ implies
\begin{align*}
\bar\nabla_XY=\nabla^{\partial}_XY+\mathbf{B}^\Sigma\left(X,Y\right)+\mathbf{B}\left(X,Y\right).
\end{align*}
On the other hand, the isometric immersions $\partial \Sigma^{n}\hookrightarrow \mathbb{S}^{2n-1}\hookrightarrow\mathbb{B}^{2n}$ gives
\begin{align*}
\bar\nabla_XY=\nabla^{\partial}_XY+\mathbf{B}^\partial\left(X,Y\right)-\hin{X}{Y}x.
\end{align*}
Thus
\begin{align*}
\mathbf{B}^\Sigma\left(X,Y\right)+\mathbf{B}\left(X,Y\right)=\mathbf{B}^\partial\left(X,Y\right)-\hin{X}{Y}x.
\end{align*}
The boundary condition \eqref{eq:nu} gives
\begin{align*}
\mathbf{B}^\Sigma\left(X,Y\right)=&-\sin\theta \hin{X}{Y}\nu,\\
\hin{\mathbf{B}\left(X,Y\right)}{J\nu}=&\cos\theta\hin{X}{Y},\\
\hin{\mathbf{B}\left(X,Y\right)}{JZ}=&\hin{\mathbf{B}^\partial\left(X,Y\right)}{JZ}.
\end{align*}

Finally, a direct calculation yields
\begin{align*}
\hin{\mathbf{B}\left(X,\nu\right)}{J\nu}=&\hin{\bar\nabla_X\nu}{J\nu}\\
=&\hin{-X(\theta)J\nu+\sin\theta X+\cos\theta JX}{J\nu}\\
=&-X(\theta).
\end{align*}
Hence
\begin{align*}
\nabla^{\partial}\theta=J\mathbf{B}\left(\nu,\nu\right).
\end{align*}
\end{proof}

 Define
\begin{align*}
\eta=\iota_{\mathbf{H}}\omega\vert_{\Sigma^{n}},\quad\eta^{\partial}=\iota_{\mathbf{H}^{\partial}}\omega\vert_{\partial\Sigma^{n}}.
\end{align*}
The one forms $\eta$ and $\eta^\partial$ are called the Maslov form of the Lagrangian immersion $\Sigma^{n}\hookrightarrow\mathbb{B}^{2n}$ and the Legendrian immersion $\partial \Sigma^{n}\hookrightarrow\mathbb{S}^{2n-1}$ respectively. Equality
\eqref{eq:1.3} implies that
\begin{align}\label{eq:1.6}
\iota_{\nu}\eta=-\hin{\mathbf{B}\left(\nu,\nu\right)}{J\nu}-(n-1)\cos\theta.
\end{align}
Equalities \eqref{eq:1.4} and \eqref{eq:1.5} yield
\begin{align}\label{eq:1.7}
\eta\vert_{\partial \Sigma^n}=\eta^\partial+\dif\theta.
\end{align}

By \eqref{eq:1.7} we obtain the following very important observation.
\begin{lem}\label{lem:2.2}
If $\Sigma^n$ is a minimal Lagrangian submanifold in $\mathbb{B}^{2n}$ with Legendrian capillary boundary, then $\partial\Sigma^n$ is a minimal Legendrian submanifold in $\mathbb{S}^{2n-1}$.
\end{lem}

\section{Main results and proofs}
In this section, we assume $$x:\Sigma\To\mathbb{B}^4$$ is a minimal Lagrangian surface with Legendrian capillary boundary on $\mathbb{S}^3$, i.e., the contact angle $\theta$ is a local constant. Then by  \autoref{lem:2.2} each component of $\partial \Sigma$ is a Legendrian  geodesic curve and hence a  Legendrian great circle in $\mathbb{S}^3$. When restricted on $\partial\Sigma$, we have from \eqref{eq:1.2} and \eqref{eq:1.6} that
\begin{align}\label{eq:3.1}
\kappa_{g}=\sin\theta,\quad \mathbf{B}\left(\nu,\nu\right)=-\cos\theta J\nu.
\end{align}
Here $\kappa_{g}$ is the geodesic curvature of the curve $\partial\Sigma$ in $\Sigma$. Let $z$ be a local conformal coordinates on $\Sigma$ and consider the cubic form $Q$ on $\Sigma$ defined by
\begin{align*}
Q=\hin{\mathbf{B}\left(\partial_{z},\partial_{z}\right)}{J\partial_{z}}\left(\dif z\right)^3.
\end{align*}
Since $\Sigma$ is minimal, we know that $Q$ is holomorphic. We have

\begin{theorem}\label{mainthm:1}
Let $\Sigma$ be a minimal Lagrangian surface in $\mathbb{B}^4$ with Legendrian free boundary on $\mathbb{S}^3$. Then $\Sigma$ is an equatorial plane disk.
\end{theorem}
\begin{proof}
If $\Sigma$ is Lagrangian submanifold with Legendrian free boundary, i.e., $\theta=\frac{\pi}{2}$, when restricted on $\partial\Sigma$, by \eqref{eq:3.1} we have
\begin{align*}
\kappa_{g}=1,\quad\mathbf{B}^\partial=0.
\end{align*}
 Hence $Q=0$ along the boundary $\partial\Sigma$, which implies that  $Q=0$ in $\Sigma$. Consequently, $\Sigma$ is totally geodesic in $\mathbb{B}^4$. Applying the Gauss-Bonnet formula we have
\begin{align*}
2\pi\left[2(1-\gamma)-r\right]=2\pi\chi\left(\Sigma\right)=\int_{\Sigma}\kappa+\int_{\partial\Sigma}\kappa_{g}=\int_{\partial\Sigma}=2\pi r,
\end{align*}
where $\kappa$ is the Gauss curvature of $\Sigma$,  $\gamma$ is the genus of $\Sigma$ and $r$ the numbers of the components of $\partial\Sigma$. Thus
\begin{align*}
\gamma+r=1.
\end{align*}
Consequently, $\gamma=0$ and $r=1$. Therefore $\Sigma$ is a topological disk and is an equatorial plane disk according to Li, Wang and Weng's result (\autoref{LWW}).
\end{proof}

In particular, we have proved the following.
\begin{cor}
There is no minimal Lagrangian annulus in $\mathbb{B}^4$ with Legendrian free boundary on $\mathbb{S}^3$.
\end{cor}

Next we will prove \autoref{conj} in the introduction. Before that, let us recall the example of Lagrangian catenoids and give some detailed descriptions on them, which will be helpful to understand our proofs presented below.

\begin{eg}[Lagrangian catenoids]\label{exam}

 We  identify a real vector  $\left(x^1,x^2,y^1,y^2\right)\in\mathbb{R}^4$ as a complex vector $\left(z^1,z^2\right)=\left(x^1+\sqrt{-1}y^1, x^2+\sqrt{-1}y^2\right)\in\mathbb{C}^2$. The Lagrangian catenoid in $\mathbb{R}^4$ can be identified as the holomorphic curve $\Sigma_{\lambda}$ in $\mathbb{C}^2$, with respect to the standard K\"ahler form $\frac{\sqrt{-1}}{2}\sum_{k=1}^2\dif z^k\wedge\dif\bar z^k$, given by
\begin{align*}
\Sigma_{\lambda}=\set{\left(z,\dfrac{\lambda}{z}\right): z\in\mathbb{C}\setminus\set{0}},
\end{align*}
where $
\lambda\in\mathbb{R}\setminus\set{0}$. Let $\Omega=\dif z^1\wedge\dif z^2$ be the holomorphic symplectic form on $\mathbb{C}^2$. Then
\begin{align*}
\Omega\vert_{\Sigma_{\lambda}}=0.
\end{align*}
Hence $\Sigma_{\lambda}$ is a Lagrangian surface in $\mathbb{C}^2$ with respect to the K\"ahler form $\mathrm{Re}\Omega$ (or $\mathrm{Im}\Omega$).
Notice that the complex structure $J$ associated with the K\"ahler form $\mathrm{Re}\Omega=\dif x^1\wedge\dif x^2-\dif y^1\wedge\dif y^2$ is
\begin{align*}
J\left(x^1,x^2,y^1,y^2\right)=\left(-x^2,x^1,y^2,-y^1\right).
\end{align*}

Let $z=re^{\sqrt{-1}\phi}$ where $(r,\phi)$ is the polar coordinates.  Then
\begin{align*}
\Sigma_{\lambda}=\set{\left(r\cos\phi,\dfrac{\lambda}{r}\cos\phi,r\sin\phi,-\dfrac{\lambda}{r}\sin\phi\right): r>0,\ 0\leq\phi<2\pi.}
\end{align*}
Set
\begin{align*}
X(r,\phi)=\left(r\cos\phi,\dfrac{\lambda}{r}\cos\phi,r\sin\phi,-\dfrac{\lambda}{r}\sin\phi\right).
\end{align*}
The tangent bundle $T\Sigma_{\lambda}$ is spanned by
\begin{align*}
X_{r}=&\left(\cos\phi,-\dfrac{\lambda}{r^2}\cos\phi,\sin\phi,\dfrac{\lambda}{r^2}\sin\phi\right),\\
X_{\phi}=&\left(-r\sin\phi,-\dfrac{\lambda}{r}\sin\phi,r\cos\phi,-\dfrac{\lambda}{r}\cos\phi\right),
\end{align*}
and the normal bundle $T^{\bot}\Sigma_{\lambda}$ is spanned by
\begin{align*}
JX_{r}=&\left(\dfrac{\lambda}{r^2}\cos\phi, \cos\phi, \dfrac{\lambda}{r^2}\sin\phi, -\sin\phi\right),\\
JX_{\phi}=&\left(\dfrac{\lambda}{r}\sin\phi, -r\sin\phi,-\dfrac{\lambda}{r}\cos\phi, -r\cos\phi\right).
\end{align*}
One can check that
\begin{align*}
\abs{X_{r}}^2-\dfrac{1}{r^2}\abs{X_{\phi}}^2=\hin{X_{r}}{X_{\phi}}=0,
\end{align*}
i.e., $X$ is a conformal immersion. Since
\begin{align*}
X_{rr}=&\left(0,\dfrac{2\lambda}{r^3}\cos\phi,0,-\dfrac{2\lambda}{r^3}\sin\phi\right),\\
X_{\phi\phi}=&\left(-r\cos\phi,-\dfrac{\lambda}{r}\cos\phi,-r\sin\phi,\dfrac{\lambda}{r}\sin\phi\right),
\end{align*}
we get
\begin{align*}
\mathbf{B}\left(X_r,X_r\right)=\dfrac{2\lambda}{r^3\abs{X_{r}}^2}JX_{r},\quad \mathbf{B}\left(X_{\phi},X_{\phi}\right)=-\dfrac{2\lambda}{r\abs{X_r}^2}JX_{r}.
\end{align*}
In particular, $\Sigma_{\lambda}$ is a minimal Lagrangian surface in $\mathbb{R}^4$.

Notice that
\begin{align*}
\hin{X_{\phi}}{JX}=0.
\end{align*}
If $0<\abs{\lambda}<\frac{1}{2}$, then $\partial\left(\Sigma_{\lambda}\cap\mathbb{B}^4\right)=\Sigma_{\lambda}\cap\mathbb{S}^3$ has two components
\begin{align*}
S_{\pm}\coloneqq\set{\left(r_{\pm}\cos\phi,\dfrac{\lambda}{r_{\pm}}\cos\phi,r_{\pm}\sin\phi,-\dfrac{\lambda}{r_{\pm}}\sin\phi\right): 0\leq\phi<2\pi},
\end{align*}
where
\begin{align*}
r_{\pm}=\sqrt{\dfrac12\pm\sqrt{\dfrac14-\lambda^2}}.
\end{align*}
These two components are Legendrian.  The unit outward normal vector field of $S_{\pm}\subset \Sigma_{\lambda}\cap\mathbb{B}^4$ is
\begin{align*}
\nu_{\pm}=&\pm\left(r_{\pm}\cos\phi,-\dfrac{\lambda}{r_{\pm}}\cos\phi,r_{\pm}\sin\phi,\dfrac{\lambda}{r_{\pm}}\sin\phi\right)\\
=&\sqrt{1-4\lambda^2}X\mp2\lambda JX.
\end{align*}
Thus, the contact angle $\theta_{\pm}$ along the boundary $S_{\pm}$ satisfies
\begin{align*}
\sin\theta_{\pm}=\sqrt{1-4\lambda^2},\quad\cos\theta_{\pm}=\mp2\lambda.
\end{align*}
In summary, $X$ is a conformal annulus minimal Lagrangian immersion from $A=\set{(r,\phi): r_{-}\leq r\leq r_+, 0\leq\phi<2\pi}$ to $\mathbb{B}^4$ with Legendrian capillary boundary on $\mathbb{S}^3$ with $X(A)=\Sigma_\lambda \ (0<\abs{\lambda}<\frac{1}{2})$. Notice that the contact angle of $\Sigma_\lambda\  (0<\abs{\lambda}<\frac{1}{2})$ can not be $\frac{\pi}{2}$.
\end{eg}

We have
\begin{theorem}\label{mainthm:2}
Assume that $\Sigma$ is an annulus type minimal Lagrangian surface in $\mathbb{B}^4$ with Legendrian capillary boundary on $\mathbb{S}^3$, then $\Sigma$ must be congruent to one of the Lagrangian catenoids $\Sigma_\lambda\ (0<\abs{\lambda}<\frac{1}{2})$.
\end{theorem}
\begin{proof}
 Assume that $\Sigma$ is given by a conformal  minimal immersion $X$ from an annulus $$A=\set{(r, \phi): r_{-}\leq r\leq r_+, 0\leq \phi< 2\pi}\subset\mathbb{R}^2$$ 
to $\mathbb{B}^4$, where we use polar coordinates $(r, \phi)$ on $A$. Denote by
$$S_{\pm}\coloneqq\set{X(r_{\pm}, \phi): 0\leq \phi< 2\pi}$$ 
 the boundary of $\Sigma$. Then
\begin{align*}
z^3\hin{\mathbf{B}\left(\partial_{z},\partial_{z}\right)}{J\partial_{z}}=\dfrac12r^3\left(\hin{\mathbf{B}\left(\partial_{r},\partial_{r}\right)}{J\partial_{r}}-\dfrac{\sqrt{-1}}{r^3}\hin{\mathbf{B}\left(\partial_{\phi},\partial_{\phi}\right)}{J\partial_{\phi}}\right)
\end{align*}
is holomorphic in $\Sigma$ and the imaginary part vanishes on $\partial\Sigma$ and hence it vanishes on $\Sigma$. Therefore this holomorphic function must be a constant by maximum principle, which can not be zero (cf. \autoref{mainthm:1}).  
 Consequently, there is a nonzero real constant $c$ such that
\begin{align*}
\mathbf{B}\left(\partial_{r},\partial_{r}\right)=\dfrac{c}{\abs{\partial_r}^2r^3}J\partial_{r}.
\end{align*}
When restricted on $\partial\Sigma=S_+\cup S_-$, according to \eqref{eq:1.3}, we have 
\begin{align*}
c=\mp r^3\abs{\partial_{r}}^3\cos\theta_{\pm}.
\end{align*}
By  \autoref{lem:2.2} we see that both $S_{\pm}$ are Legendrian geodesics , and hence are Legendrian great circles on $\mathbb{S}^3$. Consequently
\begin{align*}
c=-\cos\theta_+.
\end{align*}
Similarly, we have $c=\cos\theta_-$. Therefore 
\begin{align*}
\cos\theta_++\cos\theta_-=0,\quad  \sin\theta_+=\sin\theta_-.
\end{align*}
Let $\lambda\in (-1/2,0)\cup(0,1/2)$ be the unique real number determined by
\begin{align*}
   \sin\theta_\pm=\sqrt{1-4\lambda^2}, \cos\theta_\pm=\mp2\lambda. 
\end{align*}

Since $X$ is minimal we have $$\Delta_gX=0,$$ where $g=e^{2u}(dr^2+r^2d\phi^2)$ is a conformal metric induced on $X(A)$. Let $\Delta_0$ be the metric on the flat annulus $A$, then 
\begin{align}\label{eq:minimal}
  \Delta_0X=0.  
\end{align}
Since both $S_{\pm}$ are  Legendrian great circles on $\mathbb{S}^3$, there exist unit vectors $\vec{a}_\pm, \vec{b}_\pm\in \mathbb{R}^4$ with $\hin{\vec{a}_\pm}{ \vec{b}_\pm}=\hin{\vec{a}_\pm}{J\vec{b}_\pm}=0$, such that 
\begin{align}\label{boundary}
S_{\pm}=\vec{a}_\pm\cos\phi+\vec{b}_\pm\sin\phi.
\end{align}
  Then by applying the maximum principle to \eqref{eq:minimal} with boundary conditions \eqref{boundary}, we have
\begin{align*}
X=X(r,\phi)=\left(\vec{a}r+\dfrac{\lambda \vec{b}}{r}\right)\cos\phi+\left(\vec{c}r-\dfrac{\lambda \vec{d}}{r}\right)\sin\phi,
\end{align*}
where $\vec{a}, \vec{b}, \vec{c}, \vec{d}\in \mathbb{R}^4$ are uniquely determined by $\theta_\pm, r_\pm$ and $\vec{a}_\pm, \vec{b}_\pm$. Direct computations show that
\begin{align*}
X_{r}=&\left(\vec{a}-\dfrac{\lambda \vec{b}}{r^2}\right)\cos\phi+\left(\vec{c}+\dfrac{\lambda \vec{d}}{r^2}\right)\sin\phi,\\
X_{\phi}=&-\left(\vec{a}r+\dfrac{\lambda \vec{b}}{r}\right)\sin\phi+\left(\vec{c}r-\dfrac{\lambda \vec{d}}{r}\right)\cos\phi.
\end{align*}
Thus
\begin{align*}
\abs{X_r}^2-\dfrac{1}{r^2}\abs{X_{\phi}}^2=&\left(\abs{\vec{a}-\dfrac{\lambda \vec{b}}{r^2}}^2-\abs{\vec{c}-\dfrac{\lambda \vec{d}}{r^2}}^2\right)\cos^2\phi+\left(\abs{\vec{c}+\dfrac{\lambda \vec{d}}{r^2}}^2-\abs{\vec{a}+\dfrac{\lambda \vec{b}}{r^2}}^2\right)\sin^2\phi\\
&+2\left(\hin{\vec{a}-\dfrac{\lambda \vec{b}}{r^2}}{\vec{c}+\dfrac{\lambda \vec{d}}{r^2}}+\hin{\vec{a}+\dfrac{\lambda \vec{b}}{r^2}}{\vec{c}-\dfrac{\lambda \vec{d}}{r^2}}\right)\sin\phi\cos\phi.
\end{align*}
It follows from $\abs{X_r}^2-\frac{1}{r^2}\abs{X_{\phi}}^2=0$ that
\begin{align}\label{eq:3.5}
\abs{\vec{a}}=\abs{\vec{c}},\quad \abs{\vec{b}}=\abs{\vec{d}},\quad\hin{\vec{a}}{\vec{b}}=\hin{\vec{c}}{\vec{d}},\quad\hin{\vec{a}}{\vec{c}}=0,\quad \hin{\vec{b}}{\vec{d}}=0.
\end{align}
Then by \eqref{eq:3.5}
\begin{align*}
\hin{X_r}{\dfrac{1}{r}X_{\phi}}=&\hin{\vec{a}-\dfrac{\lambda \vec{b}}{r^2}}{\vec{c}-\dfrac{\lambda \vec{d}}{r^2}}\cos^2\phi-\hin{\vec{c}+\dfrac{\lambda \vec{d}}{r^2}}{\vec{a}+\dfrac{\lambda \vec{b}}{r^2}}\sin^2\phi\\
&+\left(\hin{\vec{c}+\dfrac{\lambda \vec{d}}{r^2}}{\vec{c}-\dfrac{\lambda \vec{d}}{r^2}}-\hin{\vec{a}-\dfrac{\lambda \vec{b}}{r^2}}{\vec{a}+\dfrac{\lambda \vec{b}}{r^2}}\right)\sin\phi\cos\phi\\
=&-\dfrac{\lambda}{r^2}\left(\hin{\vec{a}}{\vec{d}}+\hin{\vec{b}}{\vec{c}}\right),
\end{align*}
which implies from $\hin{X_r}{\frac{1}{r}X_{\phi}}=0$  that
\begin{align}\label{eq:3.6}
\hin{\vec{a}}{\vec{d}}+\hin{\vec{b}}{\vec{c}}=0.
\end{align}
Moreover
\begin{align*}
\abs{X}^2=&\abs{\vec{a}r+\dfrac{\lambda \vec{b}}{r}}^2\cos^2\phi+\abs{\vec{c}r-\dfrac{\lambda \vec{d}}{r}}^2\sin^2\phi+2\hin{\vec{a}r+\dfrac{\lambda \vec{b}}{r}}{\vec{c}r-\dfrac{\lambda \vec{d}}{r}}\sin\phi\cos\phi.
\end{align*}
When restricted on the boundary $S_\pm$ where $r=r_{\pm}$, we have $\abs{X}=1$, together with\eqref{eq:3.5} and  \eqref{eq:3.6} we get 
\begin{align}
\abs{\vec{a}}^2r_\pm^2+\frac{\lambda^2\abs{\vec{b}}^2}{r_\pm^2}&=1,\label{eq:3.7}\\
 \hin{\vec{a}}{\vec{d}}=\hin{\vec{b}}{\vec{c}}=\hin{\vec{a}}{\vec{b}}=\hin{\vec{c}}{\vec{d}}&=0. \label{eq:3.8}
\end{align}
In addition, since
\begin{align*}
r\hin{X_r}{X}=\abs{\vec{a}}^2r^2-\frac{\lambda^2\abs{\vec{b}}^2}{r^2},
\end{align*}
when restricted on the boundary $S_\pm$ where $r=r_{\pm}$ we have
\begin{align}\label{eq:3.9}
\abs{\vec{a}}^2r_\pm^2-\frac{\lambda^2\abs{\vec{b}}^2}{r_\pm^2}=\sin\theta_{\pm}.
\end{align}
By \eqref{eq:3.7} and \eqref{eq:3.9}, recall that $\sin\theta_\pm=\sqrt{1-4\lambda^2}$, we obtain that
\begin{align*}
    \abs{\vec{a}}\abs{\vec{b}}=1.
\end{align*}
Now denote $\eta=\abs{\vec{a}}>0$, by \eqref{eq:3.5}, \eqref{eq:3.8} and \eqref{eq:3.9} we have
\begin{align}\label{eqmain:1}
\abs{\vec{a}}=\abs{\vec{c}}=\eta, \abs{\vec{b}}=\abs{\vec{d}}=\frac{1}{\eta},\quad \hin{\vec{a}}{\vec{d}}=\hin{\vec{a}}{\vec{c}}=\hin{\vec{a}}{\vec{b}}=\hin{\vec{b}}{\vec{c}}=\hin{\vec{b}}{\vec{d}}=\hin{\vec{c}}{\vec{d}}=0.
\end{align}
Moreover,
\begin{align*}
\hin{X_r}{\dfrac{1}{r}JX_{\phi}}=&\hin{\vec{a}-\dfrac{\lambda \vec{b}}{r^2}}{J\vec{c}-\dfrac{J\lambda \vec{d}}{r^2}}\cos^2\phi-\hin{\vec{c}+\dfrac{\lambda \vec{d}}{r^2}}{J\vec{a}+\dfrac{J\lambda \vec{b}}{r^2}}\sin^2\phi\\
&+\left(\hin{\vec{c}+\dfrac{\lambda \vec{d}}{r^2}}{J\vec{c}-\dfrac{J\lambda \vec{d}}{r^2}}-\hin{\vec{a}-\dfrac{\lambda \vec{b}}{r^2}}{J\vec{a}+\dfrac{J\lambda \vec{b}}{r^2}}\right)\sin\phi\cos\phi.
\end{align*}
Therefore, by $\hin{X_r}{\frac{1}{r}JX_{\phi}}=0$ we obtain
\begin{align}\label{eq:3.11}
\hin{\vec{a}}{J\vec{c}}=\hin{\vec{b}}{J \vec{d}}=0,\quad \hin{\vec{a}}{J \vec{d}}+\hin{\vec{b}}{J\vec{c}}=0,\quad \hin{\vec{a}}{J \vec{b}}+\hin{\vec{c}}{J\vec{d}}=0.
\end{align}
Thus by \eqref{eq:3.11} we have
\begin{align*}
\hin{\dfrac{1}{r}X}{\dfrac{1}{r}JX_{\phi}}=&\hin{\vec{a}+\dfrac{\lambda \vec{b}}{r^2}}{J\vec{c}-\dfrac{J\lambda \vec{d}}{r^2}}\cos^2\phi-\hin{\vec{c}-\dfrac{\lambda \vec{d}}{r^2}}{J\vec{a}+\dfrac{J\lambda \vec{b}}{r^2}}\sin^2\phi\\
&+\left(\hin{\vec{c}-\dfrac{\lambda \vec{d}}{r^2}}{J\vec{c}-\dfrac{J\lambda \vec{d}}{r^2}}-\hin{\vec{a}+\dfrac{\lambda \vec{b}}{r^2}}{J\vec{a}+\dfrac{J\lambda \vec{b}}{r^2}}\right)\sin\phi\cos\phi\\
=&\dfrac{2\lambda}{r^2}\hin{\vec{b}}{J\vec{c}},
\end{align*}
which implies from $\hin{\frac{1}{r}X}{\frac{1}{r}JX_{\phi}}=0$ on $\partial \Sigma$ and \eqref{eq:3.11} that
\begin{align*}
\hin{\vec{a}}{J\vec{d}}=\hin{\vec{b}}{J\vec{c}}=0.
\end{align*}

In addition,
\begin{align*}
r\hin{X_r}{JX}=2\lambda\hin{\vec{a}}{J\vec{b}}.
\end{align*}
When restricted on the boundary $r=r_{\pm}$, since
\begin{align*}
r\hin{X_r}{JX}=\cos\theta_{\pm}=\mp2\lambda,
\end{align*}
we conclude that
\begin{align}\label{eq:3.11'}
\hin{\vec{a}}{J\vec{b}}=-1.
\end{align}

Therefore, by \eqref{eqmain:1}, \eqref{eq:3.11} and \eqref{eq:3.11'}, the real metric $O=\begin{pmatrix}\vec{a}&\vec{b}&\vec{c}&\vec{d}
\end{pmatrix}$ satisfies
\begin{align*}
O^TO=\begin{pmatrix}&\eta^2&0&0&0\\
&0&\frac{1}{\eta^2}&0&0\\
&0&0&\eta^2&0\\
&0&0&0&\frac{1}{\eta^2}
\end{pmatrix},\quad O^TJO=J=\begin{pmatrix}&0&-1&0&0\\
&1&0&0&0\\
&0&0&0&1\\
&0&0&-1&0
\end{pmatrix}.
\end{align*}
Set
\begin{align*}
    Q=\begin{pmatrix}&\eta&0&0&0\\
&0&\frac{1}{\eta}&0&0\\
&0&0&\eta&0\\
&0&0&0&\frac{1}{\eta}
\end{pmatrix}
\end{align*}
and let $P=OQ$, then we see that \begin{align*}P^TP=\mathrm{Id},\quad  P^TJP=J,
\end{align*} 
hence $P$ is a rigidity motion of $\mathbb{R}^4$ which preserves the complex structure $J$.

Finally, since $\Sigma_{\lambda}$ is invariant under the transformation $Q$, $\Sigma=O(\Sigma_{\lambda})=OQ(\Sigma_{\lambda})=P(\Sigma_{\lambda})\  (0<\abs{\lambda}<\frac{1}{2})$, we conclude that $\Sigma$ is congruent to $\Sigma_{\lambda}\  (0<\abs{\lambda}<\frac{1}{2})$. This completes the proof of   \autoref{mainthm:2}. 

\end{proof}


\end{document}